\documentclass[12pt]{article}
\usepackage{amsmath, amssymb, amsthm} 

\DeclareMathOperator*{\esssup}{ess\,sup}

\usepackage{hyperref}

\usepackage{graphicx} % Required for inserting images
\usepackage{amsmath, color, comment, mathtools, cleveref, enumitem, hyperref}
\usepackage{xcolor}

\usepackage[margin=1.2in]{geometry}

\usepackage{float}
\usepackage[maxbibnames=99, style=alphabetic]{biblatex}
\bibliography{NonLocal}
\usepackage{bm} 

\newtheorem{thm}{Theorem}[section]
\newtheorem{prop}[thm]{Proposition}
\newtheorem{lemma}[thm]{Lemma}
\newtheorem{cor}[thm]{Corollary}
\newtheorem{mydef}[thm]{Definition}

\newtheorem{remark}[thm]{Remark}

\newtheorem{claim}[thm]{Claim}
\newtheorem{assumptions}[thm]{Assumptions}

\numberwithin{equation}{section} %Equation numbering

\DeclareMathOperator{\sgn}{sgn}

\newcommand{\p}{{\partial}}

\renewcommand{\d}{\delta}

\newcommand{\R}{{\mathbb{R}}}

\newcommand{\T}{{\mathbb{T}}}
\newcommand{\g}{{\mathsf{g}}}

\newcommand{\D}{\Delta}

\newcommand{\nn}{\nonumber}

\newcommand{\ep}{\varepsilon}
\newcommand{\al}{\alpha}

\let\div\relax
\DeclareMathOperator{\div}{\mathsf{div}}

\def\XXint#1#2#3{{\setbox0=\hbox{$#1{#2#3}{\int}$ }
\vcenter{\hbox{$#2#3$ }}\kern-.6\wd0}}

\def \hal{\frac{1}{2}}
\def\({\left(}
\def\){\right)}
\def \ep{\varepsilon}
\def\nab{\nabla}

\usepackage{fancyhdr}

\title{Well-posedness of multidimensional nonlocal conservation laws with nonlinear mobility and bounded force}

\author{Antonin Chodron de Courcel\thanks{Laboratoire Alexandre Grothendieck, Institut des Hautes Études Scientifiques, Université
Paris-Saclay, CNRS, 35 Routes de Chartres, 91440 Bures-sur-Yvette, France; Email: decourcel@ihes.fr} 
}

\begin{document}

\maketitle

\begin{abstract}
We establish local-in-time existence and uniqueness results for nonlocal conservation laws with a nonlinear mobility, in several space dimensions, under weak assumptions on the kernel, which is assumed to be bounded and of finite total variation. Contrary to the linear mobility case, solutions may develop shocks in finite time, even when the kernel is smooth. We construct entropy solutions via a vanishing viscosity method, and provide a rate of convergence for this approximation scheme. 
\end{abstract}

\tableofcontents

\section{Introduction}
\subsection{Nonlinear mobilities in nonlocal conservation laws}

Nonlocal conservation laws have been the subject of many recent studies in the mathematical community. These encompass a wide variety of models and mathematical behaviors, so that the term ``nonlocal conservation law'' is actually not so precise in the end. In this article, we are interested in the Cauchy problem of such models when the \emph{mobility} function is nonlinear.

We study multidimensional scalar conservation laws of the form
\begin{equation}\label{eq:PDE}
\begin{cases}
    \p_t u + \div(f(u)\, K\ast u) = 0, & t>0, \, x\in \R^d, \\
    u\vert_{t=0} = u_0 \in L^1\cap L^\infty(\R^d).
\end{cases}
\end{equation}
Above, $d\ge 1$, $f\in C^{1,1}_{loc}(\R)$ and $K\in L^\infty\cap BV (\R^d,\R^d)$. The field $K\ast u$ will be refered to as the \emph{force field}, whereas $f(u)$ will be refered to as the \emph{mobility}. This actually is an abuse of vocabulary, since fluxes are generally written as $\mathsf{j} = u \mu(u) \mathsf{F}$, where the mobility is the function $\mu(u)$ and not $u\times \mu(u)$.

From the modelling point of view, nonlinear mobilities may represent an exclusion rule at the microscopic level: $f(u) = u(1-u)$ \cite{Giacomin1997PhaseSegregationI}. In this case, equations of the form \eqref{eq:PDE} display phase separation phenomena as in the Cahn-Hilliard model. Other nonlinear mobilities may be relevant, such as power laws $f(u) = u^m$ in the context of porous media equations \cite{Vazquez2006Porous, Carrillo2022FastRegularisation, Carrillo2022VortexFormationSuperlinear}. To name a few more applications, these models appear in the context of sedimentation \cite{Betancourt2011NonlocalSedimentation}, structured population dynamics \cite{Perthame2007Transport}, and traffic flow regulation \cite{GoatinPiccoli2024Multiscale}.

From the mathematical point of view, the equation \eqref{eq:PDE} has an hyperbolic flavor when $f$ is nonlinear. For example, when $f(u) = u^m$ and assuming that the profile $u$ varies of an order $1$ in a region of size $\ep \ll 1$, then so will the effective velocity field $u^{m-1}\, K\ast u$. This allows high density regions to move faster than low density ones, and create a shock in finite time. Moreover, the expected stability and regularity properties of entropy solutions highly depend on the kernel $K$, and it is so far not clear in the literature which condition on $K$ allows to derive, for instance, $BV$ and stability estimates.

\subsection{Related works}

As already emphasized, the nonlinear mobility case strongly departs from the linear one. In the latter case, uniqueness of weak solutions holds as soon as the kernel is as singular as $K\in BV(\R^d)$, without any entropy condition needed \cite{Coclite2022BVKernel, Colombo2024MultidimensionalBV}. 

When the mobility is nonlinear, the uniqueness of entropy solutions was proved in \cite{Betancourt2011NonlocalSedimentation}, via a Kr\"uzhkov-type argument based on $L^1$ stability, assuming $K\in C^2$, $d=1$, and $f(u) = u(1-u)^\al$, $\al \ge 1$. More recent works also assume $d=1$ and $K$ relatively smooth (e.g. $K\in C^2(\R)$ \cite{Chiarello2019Stability}). Finally, the stability estimate is sometimes obtained in a formal way, using unjustified integration by parts. We clarify the assumptions needed for the stability argument to work, namely $K\in L^\infty\cap BV(\R^d,\R^d)$ and $d\ge 1$ (no restriction on the dimension). To our knowledge, this is the first time such a result is obtained for these regularity assumptions.

Concerning the existence theory for \eqref{eq:PDE} and related models, a key argument is to provide strong compactness in order to pass to the limit in the nonlinear mobility. When no BV bound can be propagated, such a task may be difficult. For this reason, one can rely on the kinetic formulation of conservation laws \cite{PerthameDalibard2009HyperbolicKS} or show that some quantities involving singular integrals are propagated \cite{Belgacem2013Compactness, Courcel2025RepulsionModel}. Most of the time, though, BV bounds can be propagated, allowing for strong compactness. However, this property is sometimes thought to be tied with the one-dimensional situation. We clarify this point by showing that one can propagate BV bounds as soon as $K\in BV(\R^d,\R^d)$. This includes kernels as singular as Riesz flows, where $K=-\nab\g$, and $\g(x) = 1/(s|x|^s)$, up to (but not including) the Coulomb potential (so that $s<d-2$).

A natural question when building solutions to \eqref{eq:PDE} concerns the discrepancy between the approximation scheme used and the actual solution. This was first obtained by Kuznetsov \cite{Kuznetsov1976accuracy} for classical conservation laws, building on the doubling of variables method from Kr\"uzhkov. Later, the rate $O(\sqrt{\D x})$ (in the discretization of any monotone scheme) obtained by Kuznetsov was shown to be sharp \cite{TangTeng1995Sharpness} for linear conservation laws in dimension one. Nevertheless, this rate can be improved when considering genuinely nonlinear local conservation laws, see e.g. \cite{Wang1999L1Convergence}. More recently, \cite{Aggarwal2024accuracy} proved a $1/2$ rate of convergence for a model that is similar to ours, considering finite volume approximations. Their result assumes $K\in C^2$, and they must use rather involved splitting arguments if they wish to extend their result to the multidimensional case. On the opposite, we derive a $1/2$ rate of convergence for viscous approximations, in a manner that is transparent on the dimension, and for kernels $K\in L^\infty\cap BV$. To our knowledge, this is the first time such a result is obtained. 

We finally signal to the interested reader that a lot of mathematical efforts have recently been devoted to the passage from nonlocal to local models, for which we refer to \cite{Coclite2023GeneralResult, Colombo2023Overview}.

\subsection{Main results}

We summarize our main contributions as follows. 

First, we derive BV bounds under the mere assumption $ K\in BV(\R^d, \R^d)$. We then clarify the assumptions needed to derive an $L^1$ stability estimate and obtain the uniqueness of entropy solutions: we prove this result in any dimension, and for $K\in L^\infty\cap BV$. This includes many types of kernels, in particular anisotropic and non-monotone ones. Such a result is formally obtained as follows: considering two solutions $u,v$ to \eqref{eq:PDE},
\begin{align*}
\p_t |u-v| &= \div F(u,v) - \sgn(u-v) \nab f(v) \cdot K\ast (u-v) \\
&- \sgn (u-v) f(v) \div K\ast (u-v),
\end{align*}
for some $F(u,v)\in L^1$. Integrating in space gives
\begin{align*}
\frac{d}{dt} \| u-v \|_{L^1} &\le \| \nab f(v) \|_{L^1} \| K\ast (u-v) \|_{L^\infty} + \| f(v) \|_{L^\infty} \| \div K\ast (u-v ) \|_{L^1}.
\end{align*}
If $K\in L^\infty$ and $|\div K|(\R^d)<+\infty$, we can close this by a  Gr\"onwall inequality, provided $\nab f(v) \in L^1$. What is expected to hold is only $f(v) \in BV$, but if a uniform bound holds on $\| \nab f(v_\ep) \|_{L^1}$ at the level of the approximation scheme $v_\ep$, we should be able to close this loop.

We see immediately that the term $\sgn(u-v) \nab f(v) \cdot K\ast (u-v)$ is not defined when $v$ is only BV. Instead of giving a meaning to this quantity, we bound it directly at the level of the doubling of variable, using two inequalities \cref{lem:divboundA} and \cref{lem:divboundB}. 

Finally, we prove an $L^1$ rate of convergence of viscous approximations to \eqref{eq:PDE}, based on a combination of the classical argument from \cite{Kuznetsov1976accuracy} and our inequalities \cref{lem:divboundA} and \cref{lem:divboundB}.

\begin{assumptions}
Otherwise stated, we assume that $d\ge 1$, $K\in L^\infty\cap BV(\R^d,\R^d)$, and $f\in C^{1,1}_{loc}(\R)$ for which there exists some $\al >0$ and $C>0$ such that
\begin{align*}
	|f(\xi) |\le C\big( 1+ |\xi|^\al \big).
\end{align*}
\end{assumptions}

\begin{mydef}[Entropy solutions]\label[definition]{def:entropysol}
Let $d\ge 1$ and $T>0$. We say that $u\in L^\infty_{loc}((0,T); L^\infty\cap BV(\R^d))\cap C([0,T), L^1(\R^d) )$ is an entropy solution to \eqref{eq:PDE} if, for all $\eta\in C^{2}(\R)$ convex, we have in the sense of distributions on $(0,T)$
\begin{equation}
\p_t \eta(u) \le -\div(q(u) K \ast u) - (\eta'(u) f(u) - q(u))\div K \ast u,
\end{equation}
where $q' = \eta' f'$.
We say that $u$ is a \emph{global} solution if one can take $T=+\infty$ above, and that $u$ is \emph{maximal} if 
\begin{equation}
\limsup_{t\to T^-} \| u(t)\|_{L^\infty} = +\infty.
\end{equation}

\end{mydef}
\begin{remark}
Since we deal with general interaction kernels (including attractive ones) and mobilities, we cannot rule out the possibility of a finite time blow-up, for which we provide an estimation.
\end{remark}

\begin{mydef}[Vanishing viscosity solutions]\label[definition]{def:viscositysol}
Let $d\ge 1$ and $T>0$. We say that $u \in L^\infty_{loc}((0,T); L^\infty\cap BV(\R^d))\cap C([0,T), L^1(\R^d))$ is a vanishing viscosity solution to \eqref{eq:PDE} if there exists a sequence of positive reals $(\ep_k)_{k\ge 0}$ and a solution $u_k\in C^{1,1}((0,T)\times \R^d)\cap C([0,T), L^1(\R^d) )$ to 
\begin{equation}
\begin{cases}
\p_t u_k + \div(f(u_k) \nab\g\ast u_k) = \ep_k \D u_k, \\
u_k\vert_{t=0} = u_0,
\end{cases}
\end{equation}
such that $(u_k)_k$ converges to $u$ in $L^1(\R^d)$ locally uniformly in time on $[0,T)$, which we denote
\begin{align*}
u_k\xrightarrow[k\to\infty]{} u \quad \text{in } C_{loc}([0,T), L^1(\R^d)).
\end{align*}
\end{mydef}

\begin{thm} Let $d\ge 1$ and $u_0 \in L^\infty \cap BV(\R^d).$ There exists a unique maximal entropy solution to \eqref{eq:PDE}, in the sense of \cref{def:entropysol}. 
\end{thm}
\begin{cor} As a consequence, vanishing viscosity solutions are also unique, and they coincide with the notion of entropy solutions.
\end{cor}
\begin{thm} Let $d\ge 1$ and $u_0 \in L^\infty \cap BV(\R^d).$ Let $u$ be the unique maximal entropy solution to \eqref{eq:PDE} on $[0,T_{max})$, and $u_k$ a viscous approximation. Then, for all $T<T_{max}$, there exists a constant $C_T>0$ depending on $\| u_0 \|_{L^\infty\cap BV}$ such that
\begin{equation}\label{eq:rate}
\sup_{(0,T)}\| u - u_k \|_{L^1}\le C_T \sqrt{\ep_k}.
\end{equation}
\end{thm}

\begin{remark} The $BV$ regularity asked in \cref{def:entropysol} can be derived in the broader setting $K\in BV(\R^d,\R^d)$ (unbounded). This includes kernels of the form $K= - \nab\g$, with $\g$ of Riesz-type $\g(x) \sim_{|x|\to 0} 1/(s|x|^s)$, as far as $s<d-2$. The threshold $s=d-2$ corresponds to the Coulomb potential, which includes the hyperbolic Keller-Segel model \cite{PerthameDalibard2009HyperbolicKS}.
\end{remark}
\begin{remark} The $L^1$ stability result which implies the uniqueness of entropy solutions strongly relies on the assumption $K\in L^\infty(\R^d,\R^d)$. 
\end{remark}
\begin{remark} We do not think that the rate \eqref{eq:rate} is sharp in general. 
\end{remark}

\subsection{Acknowledgements}

The author thanks Beno\^it Perthame for pointing at references already covering some results of the first draft of this article, and acknowledges support from the Fondation CFM, through the Jean-Pierre Aguilar fellowship.

\section{Existence of entropy solutions}

In this section, we construct entropy solutions to \eqref{eq:PDE}, via a vanishing viscosity method. We thus consider 
\begin{equation}\label{eq:viscousPDE}
\begin{cases}
    \p_t u_\ep + \div(f(u_\ep)\, K\ast u_\ep) = \ep\D u_\ep,\\
    u_\ep\vert_{t=0} = u_0 \in L^\infty\cap BV(\R^d).
\end{cases}
\end{equation}
In the first subsection, we show that this approximate problem is locally well-posed. We chose to include this in the paper because we also prove a blowup criterion on the $L^\infty$ norm, together with an estimation of the blowup time. In the second subsection, we derive estimates which are uniform in the viscosity parameter. These estimates include $L^\infty\cap BV$ norms, and a modulus of continuity in time on solutions. We conclude this section by constructing an entropy solution and examining the time-continuity of such solutions.

\subsection{Local well-posedness for the viscous approximation}
We prove the following result.

\begin{prop}[Local well-posedness of \eqref{eq:viscousPDE}]
Let $u_0 \in L^\infty \cap L^1(\R^d).$ There exists a time $T_{max}>0$ and a unique solution $u_\ep\in L^\infty_{loc}((0,T_{max}), L^\infty\cap L^1(\R^d))$ to \eqref{eq:viscousPDE} on $[0,T_{max})$ starting at $u_0$. This solution moreover satisfies the following blow-up criterion: either $T_{max}=+\infty$, or 
\begin{align*}
\limsup_{t\to T_{max}^-} \| u_\ep (t) \|_{L^\infty\cap L^1} = +\infty.
\end{align*} 
\end{prop}
\begin{remark} The maximal time of existence $T_{max}$ established in this proposition a priori depends on $\ep$. We will actually show in the next proposition that this is not the case, thanks to the blow-up criterion and estimates that hold uniformly in the viscosity parameter. 
\end{remark}
\begin{remark}
It is outside the scope of this article to go through the regularity of the viscous solution, especially because this is a consequence of classical results from regularity theory. Indeed, since the flux is locally Lipschitz continuous, the solution $u_\ep(t,\cdot)$ must be $C^{1,1}$ for all $t>0$. This implies in particular that $\D u_\ep(t,\cdot) $ is well-defined in $L^\infty$. Note that it also holds $\nab u_\ep \in L^\infty_{loc}([0,T_{max}), L^1(\R^d)).$
\end{remark}
\begin{proof}
Let $T>0$. Consider the Banach space $X := L^\infty([0,T], L^\infty\cap L^1(\R^d))$ equipped with its natural norm, and the map $F:X\to X$ defined by
\begin{equation*}
F(v)(t,x) := e^{\ep t\D}u_0 - \int_0^t e^{\ep(t-s)\D} \div(f(v)\, K\ast v)(s,x)\, ds.
\end{equation*}
Using heat kernel estimates, we have for all $1\le p\le \infty$
\begin{align*}
\|F(v)(t)\|_{L^p} &\le \| u_0 \|_{L^p} + C\int_0^t (\ep(t-s))^{-\hal} \| f(v)\, K\ast v\|_{L^p} \, ds \\
&\le \| u_0 \|_{L^p} + C\sqrt{\frac{t}{\ep}}\esssup_{[0,T]} \|f(v)\|_{L^\infty} \|K\|_{L^1} \esssup_{[0,T]}\| v\|_{L^p}.   
\end{align*}
This implies that $F(v)\in X$, hence $F$ is well-defined as a mapping from $X$ to itself. Furthermore, fixing some $R>0$ such that $\| u_0 \|_{L^1\cap L^\infty}\le R/2 $ and considering $v\in B_R$, one obtains
\begin{align*}
\| F(v) \|_{L^p}(t) &\le \| u_0 \|_{L^p} + C\sqrt{\frac{t}{\ep}} \esssup_{-R, R} |f| \| K\|_{L^1} R \\
&\le \| u_0 \|_{L^p} + C\sqrt{\frac{t}{\ep}} (1+ R^\al) \|K \|_{L^1} R,
\end{align*}
where we have used the control on $f$ at infinity. Therefore, 
\begin{align*}
\| F(v) \|_X \le \frac{R}{2}+ C\sqrt{\frac{T}{\ep}} \| K \|_{L^1} (1+ R^\al) R.
\end{align*}
Therefore, taking $$T \le \frac{\ep}{C^2 \|K \|_{L^1}^2 (1+R^\al)^2},$$
we have $F(v) \in B_R$. Let $v,w\in B_R$. We have for all $1\le p\le \infty$,
\begin{align*}
\| F(v) -F(w)\|_{L^p}(t) &\le \int_0^t  \|e^{\ep(t-s)\D}\div( (f(v)-f(w))\, K\ast v) \|_{L^p} \, ds \\
&+ \int_0^t  \| e^{\ep(t-s)\D}\div (f(w) \, K\ast (v-w) \|_{L^p} \, ds.
\end{align*}
Using again heat kernel estimates, 
\begin{align*}
\| F(v) -F(w)\|_{L^p}(t) &\le C\int_0^t (\ep(t-s))^{-\hal } \| f(v)-f(w)\|_{L^p} \| K\ast v\|_{L^\infty } \, ds\\
&+ C\int_0^t (\ep(t-s))^{-\hal } \| f(w)\|_{L^\infty} \| K\ast (v-w)\|_{L^p}\, ds \\
&\le C \| K\|_{L^1}\| v\|_{X} \esssup_{[0,T]} (|f'(u)| + | f'(v)| ) \int_0^t (\ep(t-s))^{-\hal}\| u-w\|_{L^p}(s) \, ds  \\
&+ C \esssup_{[0,T]} (|f(u)|) \| K\|_{L^1} \int_0^t (\ep(t-s))^{-\hal}\| u-w\|_{L^p}(s) \, ds.
\end{align*}
Taking $p=1,\infty$ gives 
\begin{align*}
\| F(v) - F(w) \|_X \le C \| K \|_{L^1} \big( R \esssup_{I_R} |f'|  + \esssup_{I_R} |f|  \big) \sqrt{\frac{T}{\ep}} \| v-w\|_X,
\end{align*}
where we have defined $I_R := [-2R, 2R]$. Therefore, taking
\begin{equation}
T \le \frac{\ep}{4C^2\|K \|_{L^1}^2 \big( R \esssup_{I_R} |f'|  + \esssup_{I_R} |f|  \big)^2 },
\end{equation}
we obtain
\begin{align*}
\| F(v) - F(w) \|_X \le \hal \| v-w\|_X.
\end{align*}
Therefore, $F$ is a contraction from $B_R$ to itself. This implies in particular that there exists a fixed point $u\in B_R$, hence a solution to \eqref{eq:viscousPDE} on $[0,T_{max}]$, for some $T_{max}>0$ a priori depending on $\ep$. Our computations also provide the following stability estimate: for any two solutions $u,v\in B_R$,
\begin{align*}
 \| u-v\|_{X} \le 2 \| u_0 - v_0 \|_{L^1\cap L^\infty}.
\end{align*}
Finally, we also obtain the following blow-up criterion: considering a solution $u\in B_R$ to \eqref{eq:viscousPDE}, either $T_{max}=+\infty$, or 
\begin{align*}
\limsup_{t\to T_{max}} \|u_\ep(t)\|_{L^1\cap L^\infty} = +\infty.
\end{align*}
\end{proof} 

\subsection{Uniform estimates and existence of entropy solutions}

We prove several estimates on $u_\ep$ which do not depend on $\ep$. In particular, we obtain an estimate for the $L^\infty$ norm, which can be combined with the blow-up criterion established before to show that the maximal time of existence $T_{max}$ does not depend on $\ep$.

We also obtain BV estimates for fairly general kernels, including unbounded ones. More precisely, we only need $K\in BV(\R^d,\R^d)$. This includes any interaction as singular as Riesz flows, where $K = -\nab\g$ and $\g(x) := 1/ (s|x|^s)$, as long as $s<d-2$ (which corresponds to the Coulomb kernel).

\begin{prop}[Uniform estimates]\label[proposition]{prop:BV}
	Let $u_\ep$ be the unique solution to \eqref{eq:viscousPDE}, defined on some $[0,T_{max})$, starting from $u_0\in L^\infty\cap BV(\R^d)$. The following holds:
\begin{itemize}
\item (conservation of mass) for all $t\in [0,T_{max})$, $$\int_{\R^d} u_\ep(t,x)\, dx = \int_{\R^d} u_0(x) \, dx,$$
\item (decrease of the $L^1$ norm) for all $t\in [0,T_{max})$, $$\| u_\ep(t ) \|_{L^1 }\le \| u_0 \|_{L^1},$$
\item (local $L^\infty$ bound) there exists a universal constant $C>0$ and a time $T>0$ satisfying
	\begin{equation*}
	T \sim \frac{1}{C|\div K  |(\R^d) (1+\| u_0 \|_{L^\infty}^\al) },
	\end{equation*}
	such that
	\begin{equation}\label{eq:localLinfty}
		\forall t\in [0, T), \quad \| u_\ep(t)\|_{L^\infty} \le 2 \| u_0 \|_{L^\infty},
	\end{equation}
\item (propagation of the total variation) 
	\begin{equation}\label{eq:localTV}
		\forall t\in [0, T),\quad \int_{\R^d} |\nab u_\ep(t)| \, dx \le  e^{A(t)}TV( u_0),
	\end{equation}	  
	where $A:[0,T)\to \R$ is bounded,
	\item (continuity in time) for all $T < T_{max}$, there is a constant $C_T>0$ depending on $\| u_0\|_{L^\infty\cap BV}$ such that for all $t$ and $h$ satisfying $t+h, t \in [0,T]$, 
	\begin{equation}\label{eq:timecontinuity}
		 \| u_\ep (t+ h) - u_\ep(t) \|_{L^1}\le C_T \sqrt{h}.
	\end{equation}	 
\end{itemize}

\end{prop}
\begin{remark}
We are considering rather general interactions, in particular attractive ones. This is why the bounds we obtain here cannot be propagated for all positive times. Nevertheless, the proof naturally applies to repulsive kernels, even singular ones, for which the bounds \eqref{eq:localLinfty} and \eqref{eq:localTV} are propagated for all times. More precisely, the proof extends to kernels $K = -\nab\g$ that are of Riesz type $\g(x) \sim_{|x|\to 0} 1/ (s|x|^s)$, as long as $s<d-2$. Note that the threshold $s=d-2$ corresponds to the Coulomb kernel, which includes the hyperbolic Keller-Segel model \cite{PerthameDalibard2009HyperbolicKS}.
\end{remark}
\begin{remark} The continuity estimate \eqref{eq:timecontinuity} will not be sharp as $\ep \to 0$. Indeed, such a $1/2$-H\"older continuity estimate is typical of the regularisation by noise, and we expect a better (Lipschitz) continuity in time in the inviscid limit: this will be obtained in a second time, once the entropy solution has been constructed (see \cref{prop:timecontinuity}).
\end{remark}

\begin{proof}

\textbf{Mass conservation and decrease of the $L^1$ norm.}
The conservation of mass is straightforward. For the $L^1$ norm, we can for example write
\begin{align*}
\int_{\R^d} |u_\ep| \, dx &= \lim_{\d\to 0}\int_{\R^d} \sqrt{\d^2 + u_\ep^2} \, dx,
\end{align*}
differentiate in time, and integrate by parts, to obtain that $\| u_\ep (t) \|_{L^1} \le \| u_0 \|_{L^1}$ for all $t$ on the lifespan of $u_\ep$. 

\bigbreak
\textbf{Local $L^\infty$ bound.}
We now consider $p>1$ and compute
\begin{align*}
\frac{d}{dt}\int_{\R^d} |u_\ep|^p \, dx &= \int_{\R^d} p\sgn (u_\ep )|u_\ep|^{p-1} \big( -\div(f(u_\ep) \, K\ast u_\ep) + \ep\D u_\ep \big)\, dx \\
&= -p\int_{\R^d}  \sgn (u_\ep )|u_\ep|^{p-1}  f'(u_\ep) \nab u_\ep\cdot K \ast u_\ep \, dx - p\int_{\R^d} \sgn (u_\ep) |u_\ep|^{p-1} f(u_\ep) \div K \ast u_\ep \, dx \\
&+ \ep p \int_{\R^d} \sgn (u_\ep) |u_\ep|^{p-1} \D u_\ep \, dx.
\end{align*}
Notice that
\begin{align*}
p\sgn(u_\ep) |u_\ep|^{p-1} \D u_\ep = \D |u_\ep |^p - p(p-1) |u_\ep|^{p-2} |\nab u_\ep|^2,
\end{align*}
so that the viscous term can be discarded due to its sign. We then introduce $F'(u) := p\sgn(u) |u|^{p-1}f'(u)$, so that integrating by parts the first term gives
\begin{align*}
\frac{d}{dt}\int_{\R^d} |u_\ep|^p \, dx &\le \int_{\R^d}\big(F(u_\ep)- p\sgn(u_\ep) |u_\ep|^{p-1} f(u_\ep) \big) \div K \ast u_\ep \, dx.
\end{align*}
Integrating by parts, one obtains for $F(0) = 0$,
\begin{align*}
\forall u>0,\quad F(u) &= \int_0^u F'(\xi) \, d\xi \\
&= p |u|^{p-1} f(u) - \int_0^u p(p-1) |\xi|^{p-2} f(\xi) \, d\xi. 
\end{align*}
On the opposite, 
\begin{align*}
\forall u<0, \quad F(u) &= -p\int_0^u |\xi|^{p-1}f'(\xi) \, d\xi \\
&= -p|u|^{p-1} f(u) 
- \int_0^u p(p-1)  |\xi|^{p-2}f(\xi)\, d\xi .
\end{align*}
Overall,
\begin{align*}
	\frac{d}{dt}\int_{\R^d} |u_\ep|^p \, dx &\le - p(p-1) \int_{\R^d} \int_0^{u_\ep} |\xi|^{p-2} f(\xi)  \, d\xi \, \div K \ast u_\ep \,  dx.
\end{align*}
Thus, one obtains thanks to the control of $f$ at infinity 
\begin{align*}
\frac{d}{dt}\int_{\R^d} |u_\ep|^p \, dx &\le Cp(p-1) |\div K  |(\R^d) \|u_\ep\|_{L^\infty} \int_{\R^d} \bigg(\frac{|u_\ep|^{p-1+\al}}{p-1+\al} + \frac{|u_\ep|^{p-1}}{p-1}\bigg) \, dx \\
&\le Cp |\div K  |(\R^d) \|u_\ep\|_{L^\infty} \bigg( \frac{p-1 }{p-1+\al}\| u_\ep \|_{L^\infty}^\al   + 1\bigg) \| u_\ep \|_{L^{p-1}}^{p-1}.
\end{align*}
Interpolating between the (nonincreasing) $L^1$ norm and $L^p$ gives
\begin{align*}
\frac{d}{dt}\| u_\ep\|_{L^p}^p &\le Cp |\div K |(\R^d) \|u_\ep\|_{L^\infty} \bigg( \frac{p-1 }{p-1+\al}\| u_\ep \|_{L^\infty}^\al   + 1\bigg) \| u_0 \|_{L^1}^\frac{1}{p-1}\| u_\ep\|_{L^p}^\frac{p(p-2)}{p-1}.
\end{align*} 
Therefore,
\begin{align*}
\frac{d}{dt} \|u_\ep\|_{L^p} &= \frac{1}{p}\| u_\ep\|_{L^p}^{1-p} \frac{d}{dt} \| u_\ep\|_{L^p}^p\\
&\le C |\div K  |(\R^d) \|u_\ep\|_{L^\infty} \bigg( \frac{p-1 }{p-1+\al}\| u_\ep \|_{L^\infty}^\al   + 1\bigg) \| u_0 \|_{L^1}^\frac{1}{p-1}\| u_\ep\|_{L^p}^{-\frac{1}{p-1}}.
\end{align*}
Using that $\| u_\ep(s)\|_{L^p} \to \| u_\ep (s)\|_{L^\infty}$ as $p\to\infty$ for all $0\le s \le t$, one obtains
\begin{align*}
\frac{d}{dt}\| u_\ep \|_{L^\infty} \le  C|\div K  |(\R^d) \|u_\ep\|_{L^\infty} \big( \| u_\ep \|_{L^\infty}^\al   + 1\big).
\end{align*}
This gives the local-in-time bound for the $L^\infty$ norm. 

\bigbreak
\textbf{Total variation bound.}
For the total variation norm, we have
\begin{align*}
    \frac{d}{dt}\int_{\R^d}|\nab u_\ep|dx &= \int_{\R^d}\frac{\nab u_\ep}{|\nab u_\ep|} \cdot\nab \div(f(u_\ep)\, K\ast u_\ep ) \, dx+ \ep \int_{\R^d} \frac{\nab u_\ep}{|\nab u_\ep|} \cdot \nab \D u_\ep \, dx.
\end{align*}
For the viscous term, we again have
\begin{align*}
	\frac{\nab u_\ep}{|\nab u_\ep|} \cdot \nab \D u_\ep \le \D |\nab u_\ep|.
\end{align*}
This can thus be discarded, and one obtains
\begin{align*}
    \frac{d}{dt}\int_{\R^d}|\nab u_\ep|dx &\le  \int_{\R^d}\frac{\nab u_\ep}{|\nab u_\ep|} \cdot\nab \div(f(u_\ep)\, K\ast u_\ep ) \, dx \\
    &= \int_{\R^d} \frac{\nab u_\ep}{|\nab u_\ep|} \cdot \nab \big( \nab f(u_\ep)\cdot  K \ast u_\ep + f(u_\ep) \div K  \ast u_\ep \big) \, dx .
\end{align*}
We now develop the following term:
\begin{align*}
   & \nab \big( \nab f(u_\ep)\cdot K \ast u_\ep + f(u_\ep) \div K  \ast u_\ep \big) \\
   &= \nab\big( f'(u_\ep) \nab u_\ep\cdot K \ast u_\ep + f(u_\ep)\div K \ast u_\ep\big) \\
   &= f''(u_\ep) \nab u_\ep \nab u_\ep\cdot K \ast u_\ep + f'(u_\ep) \nab^2 u_\ep \cdot K \ast u_\ep + f'(u_\ep) \nab u_\ep \cdot\nab K \ast u_\ep \\
   &+ \nab f(u_\ep) \div K \ast u_\ep + f(u_\ep) \div K \ast \nab u_\ep.
\end{align*}
Thus,
\begin{align*}
   & \frac{\nab u_\ep}{|\nab u_\ep|} \cdot \nab \big( \nab f(u_\ep)\cdot K \ast u_\ep + f(u_\ep) \div K \ast u_\ep \big) \\
    &= f''(u_\ep)  |\nab u_\ep| \nab u_\ep\cdot K \ast u_\ep + f'(u_\ep) \nab|\nab u_\ep|\cdot K \ast u_\ep  \\
    &+ f'(u_\ep) \frac{\nab u_\ep\otimes \nab u_\ep}{|\nab u_\ep|}: \nab K \ast u_\ep + f'(u_\ep)  |\nab u_\ep| \div K \ast u_\ep + f(u_\ep) \frac{\nab u_\ep\cdot\div K \ast \nab u_\ep }{|\nab u_\ep|}.
\end{align*}
Integrating by parts, we record several cancellations and obtain in the end
\begin{multline*}
    \frac{d}{dt}\int_{\R^d} |\nab u_\ep| \, dx \le \int_{\R^d} f(u_\ep) \frac{\nab u_\ep\cdot\div K \ast \nab u_\ep }{|\nab u_\ep|} \, dx \\
    +  \int_{\R^d} f'(u_\ep) \frac{\nab u_\ep\otimes \nab u_\ep}{|\nab u_\ep|}: \nab K \ast u_\ep\, dx.
\end{multline*}
Using that $\nab K $ is a bounded operator from $L^\infty\to L^\infty$, and that $f,f'$ are locally bounded functions, one obtains
\begin{equation*}
     \frac{d}{dt}\int_{\R^d} |\nab u_\ep| \, dx \le \big(\|f(u_\ep)\|_{L^\infty}|\div K  |(\R^d)+ \|f'(u_\ep)\|_{L^\infty} |\nab K  |(\R^d)\| u_\ep \|_{L^\infty}  \big)\int_{\R^d} |\nab u_\ep| \, dx.
\end{equation*}
We now use the $L^\infty$ bound to obtain the result by Gr\"onwall's lemma.

\bigbreak
\textbf{Time continuity estimate.} We denote the flux as 
\begin{align*}
F_\ep(t,x)  := f(u_\ep ) \, K\ast u_\ep - \ep \nab u_\ep.
\end{align*} 
From what we have obtained so far, $\|F_\ep \|_{L^\infty([0,T], L^1(\R^d) )}\le C$ is uniformly bounded in $\ep>0$. Therefore, introducing a Lipschitz $\varphi$, 
\begin{align*}
\int_{\R^d} \varphi(x) (u_\ep(h,x) - u_0 (x) ) \, dx &= - \int_0^h \int_{\R^d}\varphi(x) \div F_\ep(t,x) \, dxdt \\
&= \int_0^h \int_{\R^d} \nab\varphi(x) \cdot F_\ep(t,x) \, dxdt \\
&\le C \|\nab \varphi \|_{L^\infty} h.
\end{align*}
We obtain $\| u_\ep (h ) - u_0 \|_{W^{-1,1}} \le C h$. We then interpolate between $W^{-1, 1} $ and BV to obtain
\begin{align*}
\| u_\ep (h) - u_0 \|_{L^1}& \le C \| u_\ep(h) - u_0 \|_{W^{-1,  1}}^\hal \| u_\ep (h) - u_0 \|_{W^{1,1}}^\hal \\
&\le C \sqrt{h},
\end{align*}
using the uniform bound on the total variation.

\end{proof}

We can now construct entropy solutions to \eqref{eq:PDE}. Consider a sequence $(\ep_k)_k$ such that $\ep_k \to 0$ and the unique solution $u_k$ to
\begin{align*}
\begin{cases} 
\p_t u_k + \div(f(u_k)\, K\ast u_k ) = \ep_k \D u_k, \\
u_k\vert_{t=0} = u_0\in L^\infty\cap BV(\R^d).
\end{cases}
\end{align*}
Denote $T_{max}>0$ its maximal time of existence. For any $T<T_{max}$, we have the uniform bound $\sup_{t\in (0,T) }\| u_\ep(t) \|_{BV} \le C_T$, and it is not difficult to prove that $\p_t u_\ep$ is uniformly bounded in $L^\infty([0,T], W^{-1, 1}(\R^d) )$. From the Aubin-Lions lemma, we then conclude that there exists some $u\in C([0,T], L^1(\R^d) )$ such that one can extract a subsequence -- still denoted $(u_k)_k$ -- for which $u_k \to u$ in $C([0,T], L^1(\R^d) ) $. This $u$ inherits the mass conservation, decrease of the $L^1$ norm, local $L^\infty$ and BV bounds, and the modulus of continuity in time from $u_k$. Finally, it satisfies the entropy condition. Indeed, starting from the viscous approximation, we have for all $\eta\in C^2$ and convex,
\begin{align*}
\p_t \eta(u_k) =  \eta'(u_k) \div (f(u_k) \, K \ast u_k ) + \ep_k \eta'(u_k) \D u_k.
\end{align*}
Define $q' = \eta' f'$, so that by also noticing $\D \eta (u_k) = \eta ''(u_k) |\nab u_k|^2 + \eta'(u_k) \D u_k$, we have
\begin{align*}
\p_t \eta (u_k) \le \div (q(u_k) \, K \ast u_k) + (\eta'(u_k) f(u_k) - q(u_k) ) \div K \ast u_k + \ep_k \D \eta (u_k).
\end{align*}
Passing to the limit $k\to \infty$ can be done since $f$ is locally Lipschitz, so that $f(u_k) \to f(u) $ in $C([0,T], L^1(\R^d))$. We have therefore constructed an entropy solution to \eqref{eq:PDE} in the sense of \cref{def:entropysol}.

\subsection{Time continuity}

We now return to the question of the time-continuity of entropy solutions.
\begin{prop}\label[proposition]{prop:timecontinuity}
	Let $u$ be an entropy solution to \eqref{eq:PDE}. For all $T>0$ in the lifespan of $u$, there is a constant $C_T>0$ depending on $\| u_0\|_{L^\infty\cap BV}$ such that
	\begin{equation}
	\forall (t,t+h) \in [0,T], \quad \| u(t+ h)  -  u(t) \|_{L^1} \le C_T h.
	\end{equation}
\end{prop}
\begin{proof}
	The equation reads
	\begin{align*}
	\p_t u + \div F(u) = 0,
	\end{align*}
	where $F(u) \in L^\infty([0,T], L^\infty\cap BV(\R^d))$, with 
	$$| \div F(u) |(\R^d) \le \|f'\|_{L^\infty(K) } |D u|(\R^d) \| K\|_{L^1} \|u\|_{L^\infty} + \| f\|_{L^\infty(K)} |\div K  |(\R^d)\| u \|_{L^1} =: C_T,  $$
	and $K := [-\| u\|_{L^\infty}, \| u\|_{L^\infty}].$ This implies for all $(t+h, t)\in [0,T]$,
	\begin{align*}
	\| u(t+h) - u(t) \|_{L^1} &\le \int_0^h |\p_t u (\tau) |(\R^d)\, d\tau \\
	&\le C_T h.
\end{align*}	 
\end{proof}

\section{Uniqueness of entropy solutions}

In this section, we prove by an $L^1$ stability result à la Kr\"uzhkov that entropy solutions are unique. We also show that they coincide with vanishing viscosity solutions, and give a rate of convergence for this approximation scheme, à la Kuznetsov. As emphasized in the introduction, and contrary to local conservation laws, we must be careful in the manipulation of the term
\begin{align*}
\sgn( u-v) \nab f(v) \cdot K \ast (u-v) ,
\end{align*}
which does not make sense. For this reason, we will need the following inequalities. 

\subsection{Three lemmas}

In this subsection, we prove two novel ingredients for the proof of the $L^1$ stability and rate of convergence, which are \cref{lem:divboundA} and \cref{lem:divboundB}. We also record \cref{lem:regV}.

\begin{lemma}\label[lemma]{lem:divboundA} Let $d\ge 1$. Consider $V\in L^\infty(\R^d,\R^d)$ such that $\div V\in L^1$, $a,b\in L^\infty\cap  BV(\R^d)$, $f$ locally Lipschitz, and two Lipschitz and compactly supported maps $\varphi, \psi: \R^d\to \R$. We have
\begin{multline*}
\int_{\R^d\times \R^d} \div_y\big[ \psi(x+y) \varphi(x-y)  V(x)\big] \sgn(a(x)- b(y) ) (f \circ a (x) - f\circ b(y) ) \, dxdy  \\
\le\|\psi \|_\infty  \| \varphi\|_{L^1} | D b |(\R^d) \| V\|_{L^\infty} \| f'\|_{L^\infty([-2\| b\|_{L^\infty},2 \| b\|_{L^\infty}])} .
\end{multline*}
\end{lemma}

\begin{lemma}\label[lemma]{lem:divboundB} Let $d\ge 1$. Consider $V:\R^d\to \R^d$ globally Lipschitz, $a,b\in L^\infty \cap BV(\R^d)$, $f$ locally Lipschitz, and $\psi,\varphi \in C^{0,1}_c$. There exists some $R>0$ depending on $\| b\|_{L^\infty}$ such that
\begin{multline*}
\int_{\R^d\times \R^d} \div_y \big[\psi(x+y) \varphi_\ep(x-y) \big( V(x) - V(y) \big)  \big] \sgn(a(x)- b(y) ) (f \circ a (x) - f\circ b(y) ) \, dxdy  \\
\le \| \psi \|_{\infty}\| f'\|_{L^\infty([-2\| b\|_{L^\infty},2 \| b\|_{L^\infty}])} | D b |(\R^d) \|\nab V\|_{L^\infty} M_1(\varphi),
\end{multline*}
where
\begin{align*}
M_1(\varphi) := \int_{\R^d} |x\, \varphi(x) |\, dx.
\end{align*}
\end{lemma}

\begin{proof}[Proof of \cref{lem:divboundA}]
We first assume $b\in C_c^1$. As before, the map 
\begin{align*}
	y\mapsto \sgn(a(x) - b(y) )( f \circ a(x) - f\circ b(y) ) 
	\end{align*}
	is a.e. differentiable, for a.e. $x\in \R^d$, and its derivative is 
	\begin{align*}
		y\mapsto - \sgn( a(x) - b(y) ) f'\circ b(y) \nab b(y) .
	\end{align*}
    Integrating by parts,
    \begin{align*}
   & \int_{\R^d\times \R^d} \div_y\big[ \psi(x+y) \varphi_\ep(x-y)  V(x)\big] \sgn(a(x)- b(y) ) (f \circ a (x) - f\circ b(y) ) \, dxdy \\
   &= \int_{\R^d\times \R^d} \psi(x+y) \varphi_\ep(x-y) V(x) \cdot \nab b (y) \sgn( a(x) - b(y) ) f'(b(y) ) \, dxdy \\
   &\le \|\psi\|_{\infty} \| V \|_{L^\infty} \| f'\circ b \|_{L^\infty} \int_{\R^d} |\nab b(y) | |\varphi_\ep(x - y ) | \, dxdy.
    \end{align*}
    Changing the variable in the last integral gives the bound
    \begin{align*}
    \|\psi\|_{\infty} \| V \|_{L^\infty} \| f'\circ b \|_{L^\infty} \| \varphi \|_{L^1} \| \nab b \|_{L^1}.
    \end{align*}
    Now consider that $b\in L^\infty \cap BV(\R^d)$. Therefore, there exists a sequence of smooth and compactly supported functions $(b_k)_k$ such that $b_k \to b$ in $L^1$, $\| b \|_{L^\infty}\le \sup_k \| b_k \|_{L^\infty} \le 2 \| b\|_{L^\infty} =: M$, and
	\begin{align*}
	|	 Db |(\R^d)= \lim_{k\to\infty} \|\nab b_k \|_{L^1}.
	\end{align*}
	Since $f\in C^{0,1}_{loc}$, we have that 
	\begin{align*}
	\| f'\circ b_k \|_{L^\infty} &\le \esssup_{[-\|b_k\|_{L^\infty}, \|b_k \|_{L^\infty}]} |f'| \\
	&\le \esssup_{[-M, M]} |f'| <+\infty.
	\end{align*}
	Concerning the left-hand side, the integrand is bounded -- up to taking $M$ bigger -- by
	\begin{align*}
	2|\div_y \big[\psi(x+y) \varphi_\ep(x-y) \big( V(x) - V(y) \big)  \big]  |\sup_{[-M, M]} |f|,
	\end{align*}
	which is integrable, and independent of $k$. By dominated convergence, one obtains the result. 
\end{proof}
\begin{proof}[Proof of \cref{lem:divboundB}]
	We first assume that $b\in C^1_c$. Therefore, the map 
	\begin{align*}
	y\mapsto \sgn(a(x) - b(y) )( f \circ a(x) - f\circ b(y) ) 
	\end{align*}
	is a.e. differentiable, for a.e. $x\in \R^d$, and its derivative is 
	\begin{align*}
		y\mapsto - \sgn( a(x) - b(y) ) f'\circ b(y) \nab b(y) .
	\end{align*}
    Integrating by parts,
   \begin{align*}
   & \int_{\R^d\times \R^d} \div_y \big[\psi(x+y) \varphi_\ep(x-y) \big( V(x) - V(y) \big)  \big] \sgn(a(x)- b(y) ) (f \circ a (x) - f\circ b(y) ) \, dxdy  \\
   &=  \int_{\R^d\times \R^d} \psi(x+y) \varphi_\ep (x-y) (V(x) - V(y) ) \cdot  \nab b(y)  \sgn(a(x) - b(y) )  f'(b(y) ) \, dxdy \\
   &\le \|\psi \|_{\infty} \| f'\circ b \|_{L^\infty } \int_{\R^d} dy\, |\nab b(y) |\, \int_{\R^d} dx\, | V(x)-V(y) | |\varphi_\ep(x-y)|.
   \end{align*}
   We then use the Lipschitz bound on $V$ to conclude, after a change of variable in the inner integral. We obtain the bound
   \begin{align*}
   	&\le  \|\psi \|_{\infty} \| f'\circ b \|_{L^\infty} \|\nab b \|_{L^1} \| \nab V \|_{L^\infty} M_1(\varphi).
   \end{align*}
	We treat the case $b\in L^\infty \cap BV$ as for the proof of \cref{lem:divboundA}.
\end{proof}

We finally record the following claim, whose proof is left to the reader.
\begin{claim}\label[claim]{lem:regV} Let $d\ge 1$, $K \in L^\infty(\R^d)$ and $u\in L^1(\R^d)$. Then, the field $V := K \ast u$ is bounded and uniformly continuous. If moreover $u\in BV(\R^d)$, then $V$ is globally Lipschitz continuous.
\end{claim}

\subsection{$L^1$ stability}

\begin{prop}[$L^1$ stability]\label[proposition]{prop:L1stability} Let $u, v$ be two entropy solutions to \eqref{eq:PDE} on $[0,T_1)$ and $[0,T_2)$, respectively, with initial datum $u_0, v_0 \in L^\infty\cap BV(\R^d)$. 

Then, for all $T\in (0, \min(T_1, T_2))$, there exists $C_T>0$ depending on $\|v_0 \|_{BV\cap L^\infty}$ such that
\begin{equation}
\esssup_{(0,T)} \| u - v\|_{L^1} \le C_T \| u_0 -v_0 \|_{L^1}.
\end{equation}
\end{prop}

The proof of this stability result relies on the combination of the standard doubling of variables argument and our functional inequalities \cref{lem:divboundA} and \cref{lem:divboundB}, which prevent us from giving sense to the ill-defined quantity $$\sgn(u-v) \nab f(v) \cdot K \ast (u-v)$$ arising in formal computations.

\begin{proof}
Let $u,v$ as in the statement. Given any $\eta\in C^2$ convex, one has in the sense of distributions:
    \begin{equation*}
        \p_t\eta(u) \le - \div(q(u) K\ast u) +( q(u)- f(u)\eta'(u))\div K \ast u,
    \end{equation*}
    where $q'(\xi) := f'(\xi)\eta'(\xi)$. In particular, one can justify that for any $k\in \R$,
    \begin{equation*}
        \p_t|u-k| \le -\div(\sgn(u-k) (f(u)-f(k)) K \ast u) - f(k) \sgn(u-k) \div K \ast u.
    \end{equation*}
    Therefore,
    \begin{align*}
        \p_t |u(t,x)-v(s,y)| &\le - \div_x (\sgn(u(t,x)-v(s,y))(f(u(t,x))-f(v(s,y))) K \ast u)(t,x) \\
        &- f(v(s,y)) \sgn(u(t,x)-v(s,y)) \div K \ast u(t,x), \\
        \p_s |u(t,x)-v(s,y)| &\le -\div_y(\sgn(u(t,x)-v(s,y))(f(u(t,x))-f(v(s,y))) K \ast v)(s,y) \\
        &- f(u(t,x)) \sgn(v(s,y)-u(t,x)) \div K \ast v(s,y). 
    \end{align*}
    Integrating with respect to a smooth nonnegative and compactly supported test function $\varphi \equiv \varphi(t,s,x,y)$ and summing these lines gives
    \begin{align*}
     & -\int_{[0,T] \times \R^d \times \R^d } \varphi(0, s, x, y) |u_0(x) - v(s,y) |\, dsdxdy - \int_{[0,T] \times \R^d \times \R^d } \varphi(t, 0, x, y) |u(t,x) - v_0(y) |\, dsdxdy  \\
       & - \int_{[0,T]\times [0,T] \times \R^d\times \R^d} (\p_t+\p_s)\varphi(t,s,x,y) |u(t,x)-v(s,y)|\,  dtdsdxdy \\
        &\le  \int_{[0,T]\times [0,T] \times \R^d\times \R^d}\big( \nab_x \varphi \cdot K \ast u(t,x) + \nab_y\varphi\cdot K \ast v(s,y) \big) \\
        &\quad \times  \sgn(u(t,x)-v(s,y))(f(u(t,x))-f(v(s,y)))\, dtdsdxdy \\
        &- \int_{[0,T]\times [0,T] \times \R^d\times \R^d} \varphi(t,s,x,y) \\
        &\quad \times \sgn(u(t,x)-v(s,y)) (f(v(s,y))\div K \ast u(t,x) - f(u(t,x))\div K \ast v(s,y))\, dtdsdxdy.
    \end{align*}
    We then take
    \begin{equation*}
        \varphi(t,s,x,y) := \psi\big( \frac{t+s}{2}, \frac{x+y}{2} \big) \frac{1}{\d \ep^{d}} \varphi_1\big(\frac{t-s}{\d}\big) \varphi_2\big( \frac{x-y}{\ep}\big),
    \end{equation*}
    for nonnegative, smooth, and compactly supported $\psi, \varphi_1, \varphi_2$ such that $\varphi_1, \varphi_2$ are of mass $1$. We denote $\varphi_1^\d := \d^{-1}\varphi_1(\cdot/ \d) $ and $\varphi_2^\ep := \ep^{-d} \varphi_2(\cdot / \ep)$, 
    so that
    \begin{equation*}
        \p_t \varphi +  \p_s\varphi = \varphi_1^\d\varphi_2^\ep\p_t \psi .
    \end{equation*}
    In particular, the only term involving a time derivative can be rewritten as
    \begin{align*}
    -\int_{[0,T]\times [0,T]\times \R^d\times \R^d} \p_t \psi\big(\frac{t+s}{2}, \frac{x+y}{2}\big) \varphi_2^\ep (x-y) \varphi_1^\d (t-s) |u(t,x) - v(s,y) |\, dtdsdxdy. 
    \end{align*}
    Sending $\d\to 0$ appealing to classical convergence results (such as dominated convergence, using that each integral is actually localised on a compact set), one obtains
    \begin{align*}
    &- \int_{ \R^d\times \R^d} \psi\big( 0, \frac{x+y}{2}\big) \varphi_2^\ep (x-y)  |u_0(x) - v_0(y) |\, dxdy  \\
    &- \int_{[0,T]\times \R^d\times \R^d } \p_t \psi\big( t, \frac{x+y}{2}\big) \varphi_2^\ep(x-y) |u(t,x) - v(t,y) |\, dtdxdy \\
    &\le  \int_{[0,T]\times \R^d\times \R^d } \big( \nab_x\varphi(t,t,x,y) \cdot K \ast u(t,x) + \nab_y \varphi(t,t,x,y) \cdot K \ast v(t,y) \big) \\
    &\quad \times \sgn(u(t,x) - v(t,y) ) \big[ f(u(t,x) ) - f( v(t,y) )  \big]\, dtdxdy \\
    &- \int_{[0,T]\times \R^d\times \R^d } \psi \big( t, \frac{x+y}{2}\big) \varphi^\ep_2(x-y)  \\
    &\quad \times \sgn(u(t,x) - v(t,y) ) \big[ f(v(t,y) ) \div K \ast u(t,x) - f( u(t,x) ) \div K \ast v(t,y)  \big]\, dtdxdy.
\end{align*}     
Notice that
\begin{equation*}
 \nab_x\varphi + \nab_y\varphi = \varphi_1^\d\varphi_2^\ep\nab_x \psi,
\end{equation*}
so that the right-hand side above can be written as
\begin{align}\label{eq:RHStmp}
& \int_{[0,T]\times \R^d\times \R^d } \varphi_2^\ep(x-y) \nab_x \psi\big(t, \frac{x+y}{2}\big) \cdot K \ast u(t,x) \nn \\
&\quad\times  \sgn(u(t,x) - v(t,y) ) \big[ f(u(t,x) ) - f( v(t,y) )  \big]\, dtdxdy \nn \\
&- \int_{[0,T]\times \R^d\times \R^d } \nab_y (\psi \varphi_2^\ep)(t,t,x,y) \cdot \big( K \ast u(t,x) - K \ast v(t,y) \big) \nn \\
&\quad \times \sgn(u(t,x) - v(t,y) ) \big[ f(u(t,x) ) - f( v(t,y) )  \big]\, dtdxdy \nn \\
    &- \int_{[0,T]\times \R^d\times \R^d } \psi \big( t, \frac{x+y}{2}\big) \varphi^\ep_2(x-y) \nn  \\
    &\quad \times \sgn(u(t,x) - v(t,y) ) \big[ f(v(t,y) ) \div K \ast u(t,x) - f( u(t,x) ) \div K \ast v(t,y)  \big]\, dtdxdy.
\end{align}
    Before we can identify each of the integrals above, let us develop the last term as follows:
    \begin{multline*}
          f(v(t,y))\div K \ast u(t,x) - f(u(t,x))\div K \ast v(t,y) \\
         =  f(v(t,y)) \big( \div K \ast u(t,x)-\div K \ast v(t,y) \big) -( f(u(t,x))-f(v(t,y) ) ) \div K  \ast v(t,y).
    \end{multline*}
    Thus, \eqref{eq:RHStmp} writes $I + II + III$, where
    \begin{align}
        I &:=   \int_{[0,T]\times \R^d\times \R^d } \varphi_2^\ep(x-y) \nab_x \psi\big(t, \frac{x+y}{2}\big) \cdot K \ast u(t,x) \nn \\
&\quad\times  \sgn(u(t,x) - v(t,y) ) \big[ f(u(t,x) ) - f( v(t,y) )  \big]\, dtdxdy \nn \\
        II &:= - \int_{[0,T]\times \R^d\times \R^d } \div_y \big[ \psi(t,\frac{x+y}{2}) \varphi_2^\ep (x-y)  ( K \ast u(t,x) - K \ast v(t,y) ) \big] \nn \\
&\quad \times \sgn(u(t,x) - v(t,y) ) \big[ f(u(t,x) ) - f( v(t,y) )  \big]\, dtdxdy \nn  \\
        III &:= -\int_{[0,T]\times \R^d\times \R^d } \psi \big( t, \frac{x+y}{2}\big) \varphi^\ep_2(x-y) \nn  \\
    &\quad \times \sgn(u(t,x) - v(t,y) )  f(v(t,y) ) \big[\div K \ast u(t,x) -  \div K \ast v(t,y)  \big]\, dtdxdy.
    \end{align}
    We can now identify these terms. Indeed, as $\ep\to 0$, we can again appeal to classical convergence theorems in order to obtain
    \begin{align*}
    I &\xrightarrow[\ep\to 0]{}  \int_{[0,T]\times \R^d} \nab_x \psi(t, x) \cdot K \ast u(t,x) \sgn( u-v) (t,x) \big[ f(u(t,x)) - f(v(t,x)) \big] \, dtdx \\
    &\quad = -\bigg\langle \psi, \div_x \big[ \sgn(u-v) (f(u) - f(v)) K \ast u \big] \bigg\rangle \\
    III & \xrightarrow[\ep\to 0]{} \bigg\langle \psi, \sgn(u-v) f(v) \div K \ast (u - v) \bigg\rangle.
    \end{align*}
    The difficulty lies in estimating $II$, for which we cannot appeal to standard convergence theorems since $v(t,\cdot) \in BV(\R^d)$. Using
\begin{align*}
K \ast u(t,x) - K \ast v(t,y) = K \ast (u-v)(t,x) + K \ast v(t,x) - K \ast v(t,y), 
\end{align*}    
    we decompose $II$ into $II_a + II_b$, where
\begin{align*}
II_a &= - \int_{[0,T]\times \R^d\times \R^d } \div_y \big[ \psi(t,\frac{x+y}{2}) \varphi_2^\ep (x-y) K \ast (u-v) (t,x)  \big]\\
&\times  \sgn(u(t,x) - v(t,y) ) \big[ f\circ u(t,x) - f\circ v(t,y) \big] \, dtdxdy\\
II_b &= - \int_{[0,T]\times \R^d\times \R^d } \div_y \big[ \psi(t,\frac{x+y}{2}) \varphi_2^\ep (x-y)  ( K \ast v(t,x) - K \ast v(t,y) ) \big] \\
&\times  \sgn(u(t,x) - v(t,y) ) \big[ f\circ u(t,x) - f\circ v(t,y) \big] \, dtdxdy.
\end{align*}    
     We note that $K \ast u(t)$ is globally Lipschitz continuous, for a.e. $t>0$, since $u(t,\cdot)\in BV(\R^d)$ (see \cref{lem:regV}). We thus use \cref{lem:divboundA} and \cref{lem:divboundB} to obtain, denoting $M:= 2 \| v\|_{L^\infty_{tx}}$
  \begin{align*}
  II_a &\le \int_0^T \| \psi(t)  \|_{\infty } | D v(t) |(\R^d) \| K \ast (u-v)(t) \|_{L^\infty} \|f'\|_{L^\infty([-M, M])} , \\
  II_b &\le \int_0^T \|\psi (t) \|_\infty | Dv(t)|(\R^d) \|f'\|_{L^\infty([-M, M])} \| \nab K \ast v (t) \|_{L^\infty} M_1(\varphi^\ep_2).
  \end{align*}
Since $K \ast v$ is Lipschitz continuous, sending $\ep\to 0$ yields $II_b\to 0$. We therefore obtain
\begin{align*}
	&- \int_{\R^d} \psi(0, x) |u_0 (x ) - v_0(x) |\, dx \\
	&- \int_{[0,T]\times \R^d} \p_t \psi(t,x) |u(t,x) - v(t,x) |\, dxdt \\
	&\le - \int_{[0,T]\times \R^d} \nab_x \psi(t, x) \cdot K \ast u(t,x) \sgn( u-v) (t,x) \big[ f(u(t,x)) - f(v(t,x)) \big] \, dtdx \\
	&+ \|\psi \|_\infty \esssup_{t\in (0,T)}| D v(t)|(\R^d) \|f'\|_{L^\infty([-M,M])} \int_0^T \| K \ast (u-v) \|_{L^\infty}\, dt \\
	&+ \int_{[0,T]\times \R^d } \psi  \sgn(u - v )  f(v ) \div K \ast (u- v)   \, dtdx.
\end{align*}  
  We now finish the proof by taking $\psi\to 1$, which can be done in a straightforward manner since $K \ast u\in L^\infty((0,T), L^1(\R^d))$. Finally, using the crucial bound on the force $K \in L^\infty$, we overall obtain:
  \begin{align*}
  \| u(T) - v(T) \|_{L^1} &\le \| u_0 - v_0 \|_{L^1 } + \esssup_{(0,T)}| Dv |(\R^d) \esssup_{[-M,M]}|f'| \| K \|_{L^\infty} \int_0^T \| u(t) - v(t) \|_{L^1}\, dt \\
  &+ \esssup_{[-M,M]} |f| | \div K  |(\R^d)  \int_0^T \| u(t) - v(t) \|_{L^1}\, dt.
  \end{align*}
  We conclude with the bounds of \cref{prop:BV} and by applying Gr\"onwall's lemma.
\end{proof}

\begin{cor}\label[corollary]{cor:identityentropyviscosity} Vanishing viscosity solutions and entropy solutions are the same. 
\end{cor}
\begin{proof}
Consider a vanishing viscosity solution $u$, and denote $(u_k)_k$ a viscous approximation, with viscosity sequence $(\ep_k)_k$. We then proceed as in the contruction of the entropy solution, which shows that $u$ satisfies the entropy condition. 

Conversely, suppose that $u$ is an entropy solution. Now, consider some $(\ep_k)_k$ such that $\ep_k\to 0$. There is a unique solution $u_k$ to \eqref{eq:viscousPDE}, starting from $u_0$. As in the proof of existence, we can extract a subsequence and construct a vanishing viscosity  solution $\tilde u$ from this subsequence, which satisfies the entropy condition. By the uniqueness theorem, we have $\tilde u = u$. 
\end{proof}

\subsection{$L^1$ rate of convergence}

A natural question concerns the discrepancy between the viscous approximation and the entropy solution. 
\begin{prop}[Rate of convergence for the viscous approximation]\label[proposition]{prop:rate}
	Let $u$ be the unique entropy solution to \eqref{eq:PDE}, and denote $u_\ep$ the unique solution to \eqref{eq:viscousPDE}. For all $T$ in the lifespan of $u$, there exists $C_T >0$ depending on $\|u_0 \|_{L^\infty\cap BV}$ such that
	\begin{equation}
	\forall t\in [0,T], \quad \| u(t) - u_\ep(t) \|_{L^1} \le C_T \sqrt{\ep}.
\end{equation}	 
\end{prop}

The proof of this rate of convergence combines the classical argument of Kuznetsov \cite{Kuznetsov1976accuracy} and our inequalities \cref{lem:divboundA} and \cref{lem:divboundB}. Given $\varphi\equiv \varphi(t,x,s,y)$ nonnegative, smooth and compactly supported, an entropy solution $u$ to \eqref{eq:PDE} defined on $[0,T]$, and some $v\in L^\infty((0,T), L^\infty\cap BV(\R^d) ) \cap C([0,T], L^1(\R^d) )$, we define 
\begin{multline*} 
		\D(\varphi) := -\int_{[0,T]\times (\R^d)^2} \varphi(t,x, T, y) |u(t,x) - v(T, y) |\, dtdxdy \\
		+ \int_{[0,T]\times (\R^d)^2}\varphi(t,x, 0, y) |u(t,x) - v_0 (y) |\, dtdxdy  
		+ \int_{[0,T]^2\times (\R^d)^2} \p_s \varphi(t,x,s,y ) |u(t,x ) - v(s,y) |\, dtdsdxdy \\
		+ \int_{[0,T]^2\times (\R^d)^2} \nab_y \varphi(t,x,s,y) \cdot K \ast v(s,y) \sgn(v(s,y) - u(t,x) ) (f\circ v(s,y) - f\circ u(t,x) ) \, dtdsdxdy \\
		- \int_{[0,T]^2\times (\R^d)^2} \varphi(t,x,s,y) f\circ u(t,x) \sgn( v(s,y) - u(t,x) ) \div K \ast v(s,y) \, dtdsdxdy.
	\end{multline*}
In particular, if $v$ is an entropy solution to \eqref{eq:PDE}, we have $\D(\varphi) \ge 0$. We will consider  
\begin{align*}
	\varphi(t,s,x,y) := \varphi_1^\d (x-y) \varphi_2^\eta(t-s) ,
	\end{align*}
	where $\varphi_1^\d := \d^{-d}\varphi ( \cdot / \d) $ and $\varphi_2^\eta := \eta^{-1} \varphi_2(\cdot / \eta) $, and $\varphi_1, \varphi_2$ are smooth, nonnegative and compactly supported functions of mass $1$. In this situation, we denote $\D_{\d, \eta} \equiv \D(\varphi)$.
	
\begin{lemma}[à la Kuznetsov]\label[lemma]{lem:Kuznetsov} Let $u$ be an entropy solution to \eqref{eq:PDE} defined on $[0,T]$, and $v \in L^\infty((0,T), L^\infty\cap BV(\R^d) ) \cap C([0,T], L^1(\R^d) )$ satisfying the continuity estimate \eqref{eq:timecontinuity}. Then, there exists $C_T>0$ depending on $\| u_0\|_{L^\infty\cap BV}$ and $\| v\|_{L^\infty_t(L^\infty\cap BV)_x}$, such that 
\begin{equation}\label{eq:Kuz}
\| u(T) - v(T) \|_{L^1} \le C_T \big[\| u_0 - v_0 \|_{L^1} + C_T (\d+ \sqrt{\eta}) - \inf_{(0,T)} \D_{\d,\eta} \big]  
\end{equation}
\end{lemma}
\begin{remark}
Assuming that $v$ is itself an entropy solution, we have $\D_{\d,\eta} \ge 0$. Therefore, sending $\d,\eta\to 0$ gives back the stability estimate of \cref{prop:L1stability}. Otherwise, this lemma allows to derive rates of convergence for several approximation schemes. 
\end{remark}

Letting aside the proof of this lemma for the moment, we now prove a rate of convergence for viscosity solutions. 

\begin{proof}[Proof of \cref{prop:rate}]
	Consider $u_\ep$ to be the unique solution to \eqref{eq:viscousPDE} starting from $u_0\in L^\infty\cap BV(\R^d)$. Applying \cref{lem:Kuznetsov} to $v=u_\ep$ and using the equation \eqref{eq:viscousPDE}, we obtain
	\begin{align*}
	\D_{\d,\eta} &= - \ep \int_{[0,T]^2\times (\R^d)^2} \varphi(t,x,s,y) \sgn(u_\ep(s,y) - u(t,x) ) \D u_\ep(s,y) \, dtdsdxdy.
	\end{align*}
	Since $$\sgn (u_\ep-k) \D u_\ep = \D |u_\ep - k | - \d_{u_\ep = k} |\nab u_\ep |^2,$$ in the sense of distributions, we have
	\begin{align*}
	\D_{\d,\eta} &\ge \ep \int_{[0,T]^2\times (\R^d)^2} \varphi(t,x,s,y)  \D_y |u_\ep (s,y) - u(t,x) |\, dtdsdxdy \\
	&= \ep \int_{[0,T]^2\times (\R^d)^2} \D_y\varphi(t,x,s,y)   |u_\ep (s,y) - u(t,x) |\, dtdsdxdy .
	\end{align*}
	Since $u_\ep$ is BV in space, locally uniformly in time, we obtain
	\begin{align*}
	\D_{\d,\eta} &\ge - C\frac{\ep}{\d}.
	\end{align*}
	Taking $\eta\to 0$ and optimizing over $\d$ in \eqref{eq:Kuz}, we obtain the result. 
\end{proof}

\begin{proof}[Proof of \cref{lem:Kuznetsov}]
	Consider an entropy solution $u$ to \eqref{eq:PDE}, and a function $v\in L^\infty_{loc}((0,T), L^1\cap L^\infty(\R^d)) $. We start from the entropy condition satisfied by $u$, using the doubling of variables argument:
	\begin{align*}
		&\int_{[0,T]\times \R^d\times \R^d} \varphi(T,x, s, y) |u(T, x) - v(s, y) |\, dsdxdy \\
		&-\int_{[0,T]\times [0,T]\times \R^d\times \R^d } \p_t \varphi(t,x,s,y)  |u(t,x) - v(s, y) |\, dtdsdxdy \\
		&- \int_{[0,T]\times \R^d\times \R^d} \varphi(0, x, s, y) |u_0 (x)  - v(s,y) |\, dsdxdy \\
		&\le \int_{[0,T]\times [0,T]\times \R^d\times \R^d } \nab_x\varphi(t,x,s,y) \cdot K \ast u(t,x)  \sgn( u(t,x ) - v(s, y) ) \big( f\circ u(t,x) - f\circ v(s,y) \big) \, dtds dxdy \\
		&- \int_{[0,T]\times [0,T]\times \R^d\times \R^d } \varphi(t,x,s,y) f\circ v(s,y) \sgn (u(t,x) - v(s, y)  ) \div K \ast u(t,x) \, dtdsdxdy.
	\end{align*}
	We then reverse the role played by $u$ and $v$. This gives
	\begin{align*}
		&\int_{[0,T]\times \R^d\times \R^d} \varphi(T,x, s, y) |u(T, x) - v(s, y) |\, dsdxdy +\int_{[0,T]\times (\R^d)^2} \varphi(t,x, T, y) |u(t,x) - v(T, y) |\, dtdxdy \\
		&- \int_{[0,T]\times \R^d\times \R^d} \varphi(0, x, s, y) |u_0 (x)  - v(s,y) |\, dsdxdy - \int_{[0,T]\times (\R^d)^2}\varphi(t,x, 0, y) |u(t,x) - v_0 (y) |\, dtdxdy \\
		&\le \int_{[0,T]\times [0,T]\times \R^d\times \R^d } \nab_x\varphi(t,x,s,y) \cdot \big( K \ast u(t,x) - K \ast v(s,y) \big) \\
		&\quad \times \sgn( u(t,x ) - v(s, y) ) \big( f\circ u(t,x) - f\circ v(s,y) \big) \, dtds dxdy \\
		&-  \int_{[0,T]\times [0,T]\times \R^d\times \R^d } \varphi(t,x,s,y) \big[f\circ v(s,y) \div K \ast u(t,x) - f\circ u(t,x) \div K \ast v(s,y) \big]\\
		&\quad \times \sgn (u(t,x) - v(s, y)  ) \, dtdsdxdy \\
		&- \D_{\d,\eta}(T),
	\end{align*}
	where
	\begin{multline*} 
		\D_{\d,\eta}(T) := -\int_{[0,T]\times (\R^d)^2} \varphi(t,x, T, y) |u(t,x) - v(T, y) |\, dtdxdy \\
		+ \int_{[0,T]\times (\R^d)^2}\varphi(t,x, 0, y) |u(t,x) - v_0 (y) |\, dtdxdy  
		+ \int_{[0,T]^2\times (\R^d)^2} \p_s \varphi(t,x,s,y ) |u(t,x ) - v(s,y) |\, dtdsdxdy \\
		+ \int_{[0,T]^2\times (\R^d)^2} \nab_y \varphi(t,x,s,y) \cdot K \ast v(s,y) \sgn(v(s,y) - u(t,x) ) (f\circ v(s,y) - f\circ u(t,x) ) \, dtdsdxdy \\
		- \int_{[0,T]^2\times (\R^d)^2} \varphi(t,x,s,y) f\circ u(t,x) \sgn( v(s,y) - u(t,x) ) \div K \ast v(s,y) \, dtdsdxdy.
	\end{multline*}
	Note that if $v$ is an entropy solution to \eqref{eq:PDE}, then $\D_{\d,\eta}(T) \ge 0$.
	We now use the triangle inequality as follows:
	\begin{align*}
	| u_0 (x) - v(s,y) | &\le |u_0(x) - v_0(x) | + |v_0 (x) - v_0(y) | + |v_0(y) - v(s,y) |, \\
	|u(T,x) - v(s,y) | &\ge |u(T, x) - v(T,x) | - | v(T, x) - v(T, y) | - |v(T, y) - v(s,y) |,
	\end{align*}
	and similarily for the expressions where the roles of $u$ and $v$ are interchanged.	This gives
	\begin{align}\label{tmp:rate}
	&2\| u(T) - v(T) \|_{L^1} \nn\\
	 &\le 2\| u_0 - v_0 \|_{L^1} \nn \\
	&+ \int_{\R^d} \varphi_1^\d (x-y) \big[ |v(T, x) - v(T, y) | + |u(T,x) - u(T,y )| \big] \, dxdy \nn \\
	&+ \int_{\R} \varphi_2^\eta (T-s) \big[\|v(T ) -v(s) \|_{L^1} + \| u(T) - u(s) \|_{L^1} \big] \, ds \nn \\
	&+ \int_{\R^d}  \varphi_1^\d (x-y) \big[|v_0 (x) - v_0(y) | + |u_0 (x) - u_0(y) | \big] \, dxdy \nn \\
	&+ \int_{\R} \varphi_2^\eta (-s) \big[ \| v_0 - v(s) \|_{L^1} + \| u_0 - u(s) \|_{L^1} \big]\, ds \nn \\
	& +\int_{[0,T]\times [0,T]\times \R^d\times \R^d } \nab_x\varphi(t,x,s,y) \cdot \big( K \ast u(t,x) - K \ast v(s,y) \big) \nn \\
		&\quad \times \sgn( u(t,x ) - v(s, y) ) \big( f\circ u(t,x) - f\circ v(s,y) \big) \, dtds dxdy  \\
		&-  \int_{[0,T]\times [0,T]\times \R^d\times \R^d } \varphi(t,x,s,y) \big[f\circ v(s,y) \div K \ast u(t,x) - f\circ u(t,x) \div K \ast v(s,y) \big] \nn \\
		&\quad\times \sgn (u(t,x) - v(s, y)  ) \, dtdsdxdy \nn \\
		&- \D_{\d,\eta}(T). \nn
 	\end{align}
 	Let us consider the above terms separately. First, we have for BV functions,
 	\begin{align*}
 	&\int_{\R^d}  \varphi_1^\d (x-y) |v_0 (x) - v_0(y) |\, dxdy \le C \d | Dv_0 |(\R^d), \\
 	 & \int_{\R^d}  \varphi_1^\d (x-y) |v (T,x) - v(T,y) |\, dxdy \le C \d | Dv(T) |(\R^d),
 	\end{align*}
 	and similarily for $u$. 
 	Then, we use the time continuity of $u$ and $v$ in order to obtain
 	\begin{align*}
 		& \int_{\R} \varphi_2^\eta (-s) \| u_0 - u(s) \|_{L^1}\, ds \le C_1 {\eta} , \\
 		& \int_{\R} \varphi_2^\eta (-s) \| v_0 - v(s) \|_{L^1}\, ds \le C_2 \sqrt{\eta} ,
 	\end{align*}
 	where here $C_1>0$ depends on $\| u_0 \|_{L^\infty\cap BV}$ and similarly for $C_2$, involving propagated quantities for $v$. Similar bounds hold for the other time translation terms. Then, we rewrite
 	\begin{align*}
 	& \nab_x\varphi \cdot \big( K \ast u(t,x) - K\ast v(s,y) \big) \\
 	&= \div_x \big[ \varphi   \big( K \ast u(t,x) - K \ast u(t,y) \big) \big] - \varphi  \div K \ast u(t,x) \\
 	&+ \nab_x \varphi \cdot K \ast (u-v)(t,y) + \nab_x \varphi\cdot \big[ K \ast v(t,y ) - K \ast v(s,y) \big]
 	\end{align*}
 	Noticing that $K \ast u(t)$ is bounded and globally Lipschitz for a.e. $t>0$ (\cref{lem:regV}), we can use \cref{lem:divboundB} and \cref{lem:divboundA} to bound \eqref{tmp:rate} by
 	\begin{align*}
 		& C \d - \int_{[0,T]^2\times (\R^d)^2} \varphi(t,x,s,y) \div K \ast u(t,x) \sgn(u(t,x )-v(s,y) ) \big[ f\circ u(t,x) - f\circ v(s,y ) \big]\, dtdsdxdy\\
 		& + C \int_0^T \| K \ast (u-v) (t) \|_{L^\infty}\, dt + C\int_0^T \varphi_2^\eta (t-s) \| K \ast v(t) - K \ast v(s) \|_{L^\infty}\, dt.
 	\end{align*}
 	Finally using that $K \in L^\infty$, we have obtained
 	\begin{align*}
 		&2\| u(T) - v(T) \|_{L^1} \le 2\| u_0 - v_0 \|_{L^1} + C\d + C\sqrt{\eta} 
 		+ C\int_0^T \| u(t) - v(t) \|_{L^1}\, dt \\
 		&-  \int_{[0,T]^2\times (\R^d)^2}\varphi(t,x,s,y) f\circ u(t,x) \big[\div K \ast u(t,x) - \div K \ast v(s,y) \big]   \sgn (u(t,x) - v(s, y)  ) \, dtdsdxdy \\
		&- \D_{\d,\eta }(T).
 	\end{align*}
 	Notice that the terms involving no derivatives on $\varphi$ recombine in order to give the $(\d,\eta)$-approximation of the quantity
 	\begin{align*}
 		-\int_{[0,T]\times \R^d }f(u) \div K \ast (u-v) \sgn( u-v) \, dtdx.
 	\end{align*}
 	At this point, it is not clear if this approximation can be associated with a quantitative rate of convergence, since we are dealing with not so regular kernels. We proceed as with the flux term, rewriting
 	\begin{align*}
 		 \div K \ast u(t,x) - \div K \ast v(s,y )  &=   \div K \ast (u-v) (t,x) \\
 		 &+ \div K \ast v(t,x)- \div K \ast v(s,x) \\
 		 &+ \div K \ast v(s,x) - \div K \ast v(s,y) .
 	\end{align*}
 	Now, $\div K \ast v$ inherits the time continuity and space regularity from $v$, since $\div K $ is a Radon measure. We obtain in the end
 	\begin{align*}
 		\| u(T) - v(T) \|_{L^1} \le \| u_0 - v_0 \|_{L^1} + C\d + C \sqrt{\eta} + C\int_0^T \| u(t) - v(t) \|_{L^1}\, dt - \D_{\d,\eta}(T).
 	\end{align*}
 	We conclude by applying Gr\"onwall's lemma.
\end{proof}

\section{Application to the one-dimensional CGV and hyperbolic KS models}

As a direct consequence from our analysis, we provide the uniqueness of entropy solutions to the one-dimensional hyperbolic Keller-Segel model, and answer part of a question asked by Carrillo et al. in \cite{Carrillo2022VortexFormationSuperlinear}. 

\subsection{The hyperbolic Keller--Segel model}

We consider the model
\begin{equation}\label{eq:HKS}
\begin{cases}
\p_t u + \div (u(1-u) \nab S) = 0, & t>0, \, x\in \T^d, \\
-\D S + S = u, \\
u\vert_{t=0} = u_0 \in L^\infty\cap BV(\T^d), & 0 \le u_0 \le 1.
\end{cases}
\end{equation}
posed on the $d$-dimensional torus $\T^d$, which can be identified with $[-\hal, \hal]^d$ with periodic boundary conditions. 

For general dimensions $d\ge 1$, such a singular kernel $S$ does not allow neither for the propagation of BV norms, nor for an $L^1$ stability estimate. Nevertheless, entropy solutions can be constructed using e.g. the kinetic formulation \cite{PerthameDalibard2009HyperbolicKS}. 

When $d=1$, we have $\nab S\in L^\infty\cap BV$. Therefore, we are exactly in the framework of our article, and we can state without proof the following:
\begin{cor}
Let $d=1$. There exists a unique entropy solution to \eqref{eq:HKS}.
\end{cor}

\subsection{The Carrillo--G\'omez-Castro--V\'azquez model}

We consider the model
\begin{equation}\label{eq:CGV}
\begin{cases}
\p_t u - \div (u^m \nab v) = 0, & t>0, \, x\in \T^d, \\
\displaystyle-\D v = u-\int_{\T^d} u\, dx, \\
u\vert_{t=0} = u_0 \in L^\infty\cap BV(\T^d). 
\end{cases}
\end{equation}
This model has been studied on the Euclidean space in \cite{Carrillo2022FastRegularisation} when $0<m<1$, and \cite{Carrillo2022VortexFormationSuperlinear} when $m>1$. The special case $m=1$ was already known as a model for vortices in type-II supraconductors and superfluidity \cite{E1994, CRS1996, LinZhang2000}.

In \cite{Carrillo2022FastRegularisation, Carrillo2022VortexFormationSuperlinear}, the authors restrict themselves to either $d=1$ or radial solutions. In this case, the conservation law \eqref{eq:CGV} can be seen as the derivative equation of an associated Hamilton-Jacobi equation, for which a comparison principle holds. This strategy allows to study a Cauchy problem that is simpler than the original one, eventually proving well-posedness of the Hamilton-Jacobi equation. 

However, the nonradial theory remains a challenge, with entropy solutions constructed in \cite{Courcel2025RepulsionModel} (without uniqueness). Another open problem raised in \cite{Carrillo2022VortexFormationSuperlinear} is to have a uniqueness result stated in terms of \eqref{eq:CGV}, and not the Hamilton-Jacobi equation. Going back to the $d=1$ framework, our article gives a partial answer to this issue. 

\begin{cor}
Let $d=1$, $m>0$, and $u_0 >0$. There exists a unique entropy solution to \eqref{eq:CGV}. 
\end{cor}

\printbibliography

@article{Aggarwal2024accuracy,
  author  = {Aggarwal, Aekta and Holden, Helge and Vaidya, Ganesh},
  title   = {On the accuracy of the finite volume approximations to nonlocal conservation laws},
  journal = {Numerische Mathematik},
  year    = {2024},
  volume  = {156},
  pages   = {237--271},
  doi     = {10.1007/s00211-023-01388-2},
  url     = {https://doi.org/10.1007/s00211-023-01388-2}
}

@article{CocliteKarlsenRisebro2025UpwindFiltering,
  author  = {Coclite, Giuseppe Maria and Karlsen, Kenneth H. and Risebro, Nils Henrik},
  title   = {Upwind Filtering of Scalar Conservation Laws},
  journal = {SIAM Journal on Mathematical Analysis},
  year    = {2025},
  volume  = {57},
  number  = {6},
  pages   = {6119--6143},
  doi     = {10.1137/25M1730387},
  url     = {https://doi.org/10.1137/25M1730387}
}

@article{Kuznetsov1976accuracy,
  author  = {Kuznetsov, N. N.},
  title   = {The accuracy of certain approximate methods for the computation of weak solutions of a first order quasilinear equation},
  journal = {U.S.S.R. Computational Mathematics and Mathematical Physics},
  volume  = {16},
  issue   = {6},
  pages   = {105--119},
  year    = {1976},
  doi     = {10.1016/0041-5553(76)90046-X},
  url     = {https://doi.org/10.1016/0041-5553(76)90046-X}
}

@article{ColomboGaravelloNocita2025stability,
  author  = {Colombo, Rinaldo M. and Garavello, Mauro and Nocita, Claudia},
  title   = {General stability estimates in nonlocal traffic models for several populations},
  journal = {Nonlinear Differential Equations and Applications NoDEA},
  year    = {2025},
  volume  = {32},
  number  = {34},
  doi     = {10.1007/s00030-025-01034-w},
  url     = {https://doi.org/10.1007/s00030-025-01034-w}
}

@article{Colombo2023Overview,
  author  = {Colombo, Maria and Crippa, Gianluca and Marconi, Elio and Spinolo, Laura V.},
  title   = {An overview on the local limit of non-local conservation laws, and a new proof of a compactness estimate},
  journal = {Journ\'ees \'equations aux d\'eriv\'ees partielles},
  year    = {2023},
  pages   = {1--14},
  note    = {Expos\'e n\textsuperscript{o} 10},
  doi     = {10.5802/jedp.681},
  url     = {https://doi.org/10.5802/jedp.681}
}

@article{Wang1999L1Convergence,
  author  = {Wang, Wei-Cheng},
  title   = {On {L\textsuperscript{1}} Convergence Rate of Viscous and Numerical Approximate Solutions of Genuinely Nonlinear Scalar Conservation Laws},
  journal = {SIAM Journal on Mathematical Analysis},
  volume  = {30},
  number  = {4},
  pages   = {780--802},
  year    = {1999},
  doi     = {10.1137/S0036141097316408},
  url     = {https://doi.org/10.1137/S0036141097316408}
}

@article{TangTeng1995Sharpness,
  author  = {Tang, Tao and Teng, Zhen-huan},
  title   = {The sharpness of {Kuznetsov}'s $\mathcal{O}(\sqrt{\Delta x})\, L^1$-error estimate for monotone difference schemes},
  journal = {Mathematics of Computation},
  volume  = {64},
  number  = {210},
  pages   = {581--589},
  year    = {1995},
  doi     = {10.1090/S0025-5718-1995-1270625-9},
  url     = {https://doi.org/10.1090/S0025-5718-1995-1270625-9}
}

@article{Coclite2023GeneralResult,
  author  = {Coclite, Giuseppe Maria and Coron, Jean-Michel and De Nitti, Nicola and Keimer, Alexander and Pflug, Lukas},
  title   = {A general result on the approximation of local conservation laws by nonlocal conservation laws: The singular limit problem for exponential kernels},
  journal = {Annales de l'Institut Henri Poincar\'e C, Analyse Non Lin\'eaire},
  year    = {2023},
  volume  = {40},
  issue   = {5},
  pages   = {1205--1223},
  doi     = {10.4171/AIHPC/58},
  url     = {https://doi.org/10.4171/AIHPC/58}
}

@article{BlandinGoatin2016Wellposedness,
  author  = {Blandin, Sebastien and Goatin, Paola},
  title   = {Well-posedness of a conservation law with non-local flux arising in traffic flow modeling},
  journal = {Numerische Mathematik},
  year    = {2016},
  volume  = {132},
  number  = {2},
  pages   = {217--241},
  doi     = {10.1007/s00211-015-0717-6},
  url     = {https://doi.org/10.1007/s00211-015-0717-6}
}

@article{Ciaramaglia2025MulticlassNonlocal,
  author  = {Ciaramaglia, Ilaria and Goatin, Paola and Puppo, Gabriella},
  title   = {A multi-class non-local macroscopic model with time delay for mixed autonomous / human-driven traffic},
  journal = {arXiv preprint arXiv:2501.09440},
  year    = {2025},
  doi     = {10.48550/arXiv.2501.09440},
  eprint  = {2501.09440},
  archiveprefix = {arXiv},
  primaryclass = {math.AP}
}

@article{GoatinPiccoli2024Multiscale,
  author  = {Goatin, Paola and Piccoli, Benedetto},
  title   = {A multi-scale multi-lane model for traffic regulation via autonomous vehicles},
  journal = {arXiv preprint arXiv:2404.17192},
  year    = {2024},
  doi     = {10.48550/arXiv.2404.17192},
  eprint  = {2404.17192},
  archiveprefix = {arXiv},
  primaryclass = {math.AP}
}

@article{PerthameDalibard2009HyperbolicKS,
  author  = {Perthame, Benoît and Dalibard, Anne-Laure},
  title   = {Existence of solutions of the hyperbolic {Keller-Segel} model},
  journal = {Transactions of the American Mathematical Society},
  volume  = {361},
  number  = {5},
  pages   = {2319--2335},
  year    = {2009},
  doi     = {10.1090/S0002-9947-08-04656-4},
  url     = {https://doi.org/10.1090/S0002-9947-08-04656-4}
}

@article{Carrillo2022FastRegularisation,
  author  = {Carrillo, José A. and Gómez-Castro, David and Vázquez, Juan Luis},
  title   = {A fast regularisation of a {Newtonian} vortex equation},
  journal = {Annales de l'Institut Henri Poincaré C, Analyse Non Linéaire},
  year    = {2022},
  volume  = {39},
  number  = {3},
  pages   = {705--747},
  doi     = {10.4171/AIHPC/17},
  url     = {https://doi.org/10.4171/AIHPC/17}
}

@article{Carrillo2022VortexFormationSuperlinear,
  author  = {Carrillo, José A. and Gómez-Castro, David and Vázquez, Juan Luis},
  title   = {Vortex formation for a non-local interaction model with {Newtonian} repulsion and superlinear mobility},
  journal = {Advances in Nonlinear Analysis},
  year    = {2022},
  volume  = {11},
  number  = {1},
  pages   = {937--967},
  doi     = {10.1515/anona-2021-0231},
  url     = {https://doi.org/10.1515/anona-2021-0231}
}

@article{AggarwalColomboGoatin2015NonlocalSystems,
  author  = {Aggarwal, Aekta and Colombo, Rinaldo M. and Goatin, Paola},
  title   = {Nonlocal Systems of Conservation Laws in Several Space Dimensions},
  journal = {SIAM Journal on Numerical Analysis},
  year    = {2015},
  volume  = {53},
  number  = {2},
  pages   = {963--983},
  doi     = {10.1137/140975255},
  url     = {https://doi.org/10.1137/140975255}
}

@article{Colombo2024MultidimensionalBV,
  author  = {Colombo, Maria and Crippa, Gianluca and Spinolo, Laura V.},
  title   = {On multidimensional nonlocal conservation laws with {BV} kernels},
  journal = {arXiv preprint arXiv:2408.02423},
  year    = {2024},
  doi     = {10.48550/arXiv.2408.02423},
  eprint  = {2408.02423},
  archiveprefix = {arXiv},
  primaryclass = {math.AP},
  url     = {https://arxiv.org/abs/2408.02423}
}

@article{Coclite2022BVKernel,
  author  = {Coclite, Giuseppe Maria and De Nitti, Nicola and Keimer, Alexander and Pflug, Lukas},
  title   = {On existence and uniqueness of weak solutions to nonlocal conservation laws with {BV} kernels},
  journal = {Zeitschrift f\"ur Angewandte Mathematik und Physik},
  year    = {2022},
  volume  = {73},
  issue   = {6},
  pages   = {241},
  doi     = {10.1007/s00033-022-01766-0},
  url     = {https://doi.org/10.1007/s00033-022-01766-0}
}

@article{DiFrancesco2025MeasureSolutions,
  author  = {Di Francesco, Marco and Fagioli, Simone and Radici, Emanuela},
  title   = {Measure solutions, smoothing effect, and deterministic particle approximation for a conservation law with nonlocal flux},
  journal = {Annales de l'Institut Henri Poincaré C, Analyse Non Linéaire},
  year    = {2025},
  volume  = {42},
  number  = {1},
  pages   = {1--43},
  doi     = {10.4171/AIHPC/150},
  url     = {https://doi.org/10.4171/AIHPC/150}
}

@article{Giacomin1997PhaseSegregationI,
  author  = {Giacomin, Giambattista and Lebowitz, Joel L.},
  title   = {Phase segregation dynamics in particle systems with long range interactions. {I}. Macroscopic limits},
  journal = {Journal of Statistical Physics},
  year    = {1997},
  volume  = {87},
  number  = {1},
  pages   = {37--61},
  doi     = {10.1007/BF02181479},
  url     = {https://doi.org/10.1007/BF02181479}
}

@article{Betancourt2011NonlocalSedimentation,
  author  = {Betancourt, Fernando and Bürger, Raimund and Karlsen, Kenneth H. and Tory, Elmer M.},
  title   = {On nonlocal conservation laws modelling sedimentation},
  journal = {Nonlinearity},
  year    = {2011},
  volume  = {24},
  number  = {3},
  pages   = {855--885},
  doi     = {10.1088/0951-7715/24/3/008},
  url     = {https://doi.org/10.1088/0951-7715/24/3/008}
}

@book{Perthame2007Transport,
  author    = {Perthame, Benoît},
  title     = {Transport Equations in Biology},
  series    = {Frontiers in Mathematics},
  publisher = {Birkhäuser Verlag},
  address   = {Basel},
  year      = {2007},
  isbn      = {978-3-7643-7740-3}
}

@article{Chiarello2019Stability,
title = {Stability estimates for non-local scalar conservation laws},
journal = {Nonlinear Analysis: Real World Applications},
volume = {45},
pages = {668-687},
year = {2019},
issn = {1468-1218},
doi = {https://doi.org/10.1016/j.nonrwa.2018.07.027},
url = {https://www.sciencedirect.com/science/article/pii/S1468121818307788},
author = {Felisia Angela Chiarello and Paola Goatin and Elena Rossi}
}

@article{Belgacem2013Compactness,
title = {Compactness for nonlinear continuity equations},
journal = {Journal of Functional Analysis},
volume = {264},
number = {1},
pages = {139-168},
year = {2013},
issn = {0022-1236},
doi = {https://doi.org/10.1016/j.jfa.2012.10.005},
url = {https://www.sciencedirect.com/science/article/pii/S0022123612003734},
author = {Fethi Ben Belgacem and Pierre-Emmanuel Jabin}
}

@article{Courcel2025RepulsionModel,
  author  = {Chodron de Courcel, Antonin and Elbar, Charles},
  title   = {On a repulsion model with {Coulomb} interaction and nonlinear mobility},
  journal = {arXiv preprint arXiv:2510.16894},
  year    = {2025},
  eprint  = {2510.16894},
  archiveprefix = {arXiv},
  primaryclass = {math.AP},
  url     = {https://arxiv.org/abs/2510.16894}
}

@book{Vazquez2006Porous,
  author    = {Vázquez, Juan Luis},
  title     = {The {Porous} {Medium} {Equation}: Mathematical Theory},
  series    = {Oxford Mathematical Monographs},
  publisher = {Oxford University Press},
  address   = {Oxford},
  year      = {2006},
  doi       = {10.1093/acprof:oso/9780198569039.001.0001},
  url       = {https://doi.org/10.1093/acprof:oso/9780198569039.001.0001}
}

@article{LinZhang2000,
title = {On the hydrodynamic limit of Ginzburg-Landau vortices},
journal = {Discrete and Continuous Dynamical Systems},
volume = {6},
number = {1},
pages = {121-142},
year = {2000},
issn = {1078-0947},
doi = {10.3934/dcds.2000.6.121},
url = {https://www.aimsciences.org/article/id/c0c5bc44-b7ec-4b11-bbe5-a3dd12f1427d},
author = {Fanghua Lin and Ping Zhang}
}

@article{CRS1996, 
	title={A mean-field model of superconducting vortices}, 
	volume={7}, DOI={10.1017/S0956792500002242}, 
	number={2}, 
	journal={European Journal of Applied Mathematics}, 
	author={Chapman, S. J. and Rubinstein, J. and Schatzman, M.}, 
	year={1996}, 
	pages={97–111}
}

@article{E1994,
  title = {Dynamics of vortex liquids in Ginzburg-Landau theories with applications to superconductivity},
  author = {E, Weinan},
  journal = {Phys. Rev. B},
  volume = {50},
  issue = {2},
  pages = {1126--1135},
  numpages = {0},
  year = {1994},
  month = {Jul},
  publisher = {American Physical Society},
  doi = {10.1103/PhysRevB.50.1126},
  url = {https://link.aps.org/doi/10.1103/PhysRevB.50.1126}
}
\nocite{*}

\end{document}